\documentclass[10pt,a4paper]{article}
\usepackage[a4paper]{geometry}
\usepackage{amssymb,latexsym,amsmath,amsfonts,amsthm}
\usepackage{graphicx,comment}
\usepackage{epsfig}
\usepackage{tikz}
\usepackage[T1]{fontenc}

\newcommand{\tr}{\mathrm{tr}}

\newtheorem{theorem}{Theorem}[section]
\newtheorem{lemma}[theorem]{Lemma}

\newtheorem{corollary}[theorem]{Corollary}

\theoremstyle{definition}
\newtheorem{definition}[theorem]{Definition}

\theoremstyle{remark}

\newtheorem{remark}[theorem]{Remark}

\numberwithin{equation}{section}

\title{Characteristic polynomials of random banded Hessenberg matrices and Hermite-Pad\'{e} approximation}

\author{A. L\'{o}pez-Garc\'{i}a \qquad V. A. Prokhorov}

\date{\today}

\begin{document}

\maketitle

\begin{abstract}
We consider a class of random banded Hessenberg matrices with independent entries having identical distributions along diagonals. The distributions may be different for entries belonging to different diagonals. For a sequence of $n\times n$ matrices in the class considered, we investigate the asymptotic behavior of their empirical spectral distribution as $n$ tends to infinity.\end{abstract}

\section{Introduction}
In a variety of important problems in analysis and probability, one considers a sequence of polynomials $\{p_{n}\}_{n=0}^{\infty}$, $\deg(p_{n})=n$, that satisfy a high-order difference equation of the form
\begin{equation}\label{intro:recrel}
z p_{n}=a_{n,n+1} p_{n+1}+a_{n,n} p_{n}+\cdots+a_{n,n-r} p_{n-r},\qquad n\geq r,
\end{equation}    
for some fixed $r\geq 1$, and one wishes to deduce an asymptotic property of these polynomials, for instance the limiting distribution of their zeros. Frequently, a useful approach for such problem is to identify these polynomials, or a normalization of them, as characteristic polynomials of a sequence of matrices, which in the case of \eqref{intro:recrel} will have the structure of banded Hessenberg matrices. 

There are several important families of polynomials that satisfy a relation of the form \eqref{intro:recrel}. Besides the classical example of orthogonal polynomials on the real line in the case $r=1$, we find classes of polynomials that satisfy orthogonality conditions with respect to a collection of $r\geq 2$ measures in the complex plane. These include the multiorthogonal polynomials associated with Angelesco and Nikishin systems of measures, supported on the real line \cite{Ang,Nik,NikSor,AptKalLopRoc} or on symmetric starlike sets in the complex plane \cite{AptKalSaff,DelLop,LeuVan1,LeuVan2,Lop,LopMin}. We also find more general classes of polynomials satisfying orthogonality conditions with respect to systems of abstract linear functionals.

An important property of Angelesco and Nikishin polynomials, equivalent to the multiorthogonality conditions, is that they appear as denominators of simultaneous Pad\'{e} approximants (commonly known as Hermite-Pad\'{e} approximants) for systems of analytic functions. Interestingly, it was proved by Kalyagin \cite{Kal} that polynomials satisfying \eqref{intro:recrel} can be realized as denominators of Hermite-Pad\'{e} approximants for a system of Weyl (resolvent) functions of the banded Hessenberg operator constructed from the coefficients in \eqref{intro:recrel}. This property is stated in precise terms in Theorem~\ref{theo:Kal}, and it is essential for our analysis. Numerous other works have investigated this Hermite-Pad\'{e} approximation property, see e.g. \cite{AptKal, AptKalVan, AptKalSaff, BarGerLop, BarLopMarTor, DelLopLop, RobSan, VanIseg}. 

We describe now the problem investigated in this work, concerning a class of \emph{random} banded Hessenberg operators. From now on $p\geq 1$ will denote a fixed but arbitrary positive integer. Let $\mu_{k}$, $0\leq k\leq p$, be a collection of $p+1$ Borel probability measures with compact support in the complex plane. For each $0\leq k\leq p$, let $a^{(k)}=(a_{n}^{(k)})_{n=1}^{\infty}$ be a sequence of complex i.i.d. random variables with distribution $\mu_{k}$. We also assume that the whole collection $\{a_{n}^{(k)}: n\geq 1,\,\,0\leq k\leq p\}$ is jointly independent. To simplify matters, in this work we will assume that the sequences of random variables considered are surely bounded in modulus by an absolute constant. With the $p+1$ sequences $(a_{n}^{(k)})_{n=1}^{\infty}$, $0\leq k\leq p$, we construct the infinite matrix $H=(h_{i,j})_{i,j=1}^{\infty}$ with entries
\begin{equation}\label{def:osH:intro}
\begin{cases}
h_{j-1,j}=1, & j\geq 2,\\
h_{j+k,j}=a_{j}^{(k)}, & 0\leq k\leq p, \quad j\geq 1,\\
h_{i,j}=0, & \mbox{otherwise}.
\end{cases}
\end{equation}
Note that this is a banded lower Hessenberg matrix. We denote by $H_{n}$ the principal $n\times n$ truncation of $H$, that is
\begin{equation}\label{def:Hn:intro}
H_{n}=\begin{pmatrix}
a_{1}^{(0)} & 1 & & & 0 \\
\vdots & \ddots & \ddots \\
a_{1}^{(p)} & & \ddots & \ddots & \\
 & \ddots & & \ddots & 1 \\
0 &  & a_{n-p}^{(p)} & \cdots & a_{n}^{(0)}
\end{pmatrix}.
\end{equation}
Let $\{\lambda_{i,n}\}_{i=1}^{n}$ denote the eigenvalues of $H_{n}$, counting multiplicities, and let
\begin{equation}\label{empmeas}
\sigma_{n}:=\frac{1}{n}\sum_{i=1}^{n}\delta_{\lambda_{i,n}}.
\end{equation}
Since we have uniform boundedness of the matrix entries, the eigenvalues $\lambda_{i,n}$ are also uniformly bounded. Clearly, $\sigma_{n}$ is a random probability measure. Its mean $\mathbb{E}\sigma_{n}$ is the probability measure defined via duality by
\[
\int f\,d\mathbb{E}\sigma_{n}=\mathbb{E} \int f\,d\sigma_{n}
\]
for every continuous function $f$.

We address in this work the following questions: Is the sequence of mean measures $\mathbb{E}\sigma_{n}$ weakly convergent? If so, what is the limit and how is it related to the distributions $\{\mu_{k}\}_{k=0}^{p}$? We provide a partial answer to these questions, proving the existence of the limits \eqref{asympformintro} and describing a generating function for these limits.

Let $Q_{n}(z):=\det(z I_{n}-H_{n})$ be the characteristic polynomial of the matrix $H_{n}$. Expanding the determinant along the last row, we see that the sequence $(Q_{n}(z))_{n=0}^{\infty}$ can be defined as the solution $(y_{n})_{n=0}^{\infty}$ of the difference equation 
\[
z y_{n}=y_{n+1}+a^{(0)}_{n+1}y_{n}+a^{(1)}_{n} y_{n-1}+\cdots+a^{(p)}_{n-p+1}y_{n-p},\qquad n\geq p,
\] 
with initial conditions $y_{k}(z)=\det(z I_{k}-H_{k})$, $0\leq k\leq p$.

We describe now our main result. Let $\{e_{n}\}_{n\in\mathbb{N}}$ denote the standard orthonormal basis in $\ell^{2}(\mathbb{N})$, and let $\mathcal{H}$ be the bounded operator on $\ell^{2}(\mathbb{N})$ whose matrix representation in this basis is the infinite matrix $H$ in \eqref{def:osH:intro}, i.e., the operator satisfying $\langle \mathcal{H} e_{j},e_{i}\rangle=h_{i,j}$, $i, j\in\mathbb{N}$. Let
\begin{equation}\label{Weylfunc:intro}
\phi_{j}(z):=\langle(zI-\mathcal{H})^{-1} e_{j},e_{1}\rangle,\qquad 1\leq j\leq p.
\end{equation}
We also define $\phi_{0}\equiv 1$. We have thus obtained a relation
\[
\mathcal{A}=(a^{(0)}, a^{(1)}, \ldots, a^{(p)})\longmapsto \Phi=(\phi_{0}, \phi_{1},\ldots,\phi_{p})
\]
which we will indicate by writing $\Phi=\Phi(\mathcal{A})$.

In this paper we will frequently use the following notation. If $f(z)$ is a Laurent series of the form $f(z)=\sum_{k\in\mathbb{Z}} c_{k} z^{-k}$, we write $c_{k}=[f]_{k}$. Our main result is the following:

\begin{theorem}\label{theo:main:intro}
Let $\mathcal{A}=(a^{(0)},\ldots,a^{(p)})$ and $\mathcal{B}=(b^{(0)},\ldots,b^{(p)})$ be two independent collections of random sequences with corresponding distributions $(\mu_{0},\ldots,\mu_{p})$, as indicated above. Let $\Phi(\mathcal{A})=(\phi_{0},\ldots,\phi_{p})$ and $\Phi(\mathcal{B})=(\psi_{0},\ldots,\psi_{p})$ be the associated systems of Weyl functions. Further, let $\alpha=(\alpha_{j}^{(k)})_{0\leq j\leq k\leq p}$ be a triangular array of independent random variables, where $\alpha_{j}^{(k)}$ has distribution $\mu_{k}$ for all $j$ and $k$, and such that $\alpha$ is independent of $\mathcal{A}$ and $\mathcal{B}$. Let $\sigma_{n}$ be the empirical measure \eqref{empmeas} of the matrix $H_{n}$ defined in \eqref{def:Hn:intro}. For each $\ell\in\mathbb{Z}_{\geq 0}$,
\begin{equation}\label{asympformintro}
\lim_{n\rightarrow\infty}\mathbb{E}\left(\int x^{\ell}\,d\sigma_{n}(x)\right)=\mathbb{E}([W]_{\ell+1}),
\end{equation}
where
\begin{equation}\label{funcW:intro}
W(z)=\frac{1}{z-\sum_{k=0}^{p}\sum_{j=0}^{k}\alpha_{j}^{(k)}\,\phi_{k-j}(z)\,\psi_{j}(z)}.
\end{equation} 
\end{theorem} 

From this result we deduce the following: If the eigenvalues of $H_{n}$ are all real for every $n$, then the mean measure $\mathbb{E}\sigma_{n}$ converges weakly to a probability distribution on the real line.

Since we have uniform boundedness of the supports of the measures $\mathbb{E} \sigma_{n}$, an equivalent formulation of \eqref{asympformintro} is that for all $z$ large enough,
\[
\lim_{n\rightarrow\infty}\mathbb{E}\left(\int\frac{1}{z-x}\,d\sigma_{n}(x)\right)=\mathbb{E}(W(z)).
\]
We will show that the function $W(z)$ can also be realized as a Weyl function of a two-sided bounded operator on $\ell^{2}(\mathbb{Z})$ defined as follows. Let $\{\hat{e}_{n}\}_{n\in\mathbb{Z}}$ be the standard basis in $\ell^{2}(\mathbb{Z})$. Take $p+1$ sequences of random variables $(a_{n}^{(k)})_{n\in\mathbb{Z}}$, with distribution $\mu_{k}$ and jointly independent, and let $\mathcal{M}$ be the operator on $\ell^{2}(\mathbb{Z})$ that satisfies
\[
\mathcal{M}\hat{e}_{n}=\hat{e}_{n-1}+\sum_{k=0}^{p} a_{n}^{(k)} \hat{e}_{n+k},\qquad n\in\mathbb{Z}.
\]
Then we have
\[
W(z)=\langle(zI-\mathcal{M})^{-1}\hat{e}_{0},\hat{e}_{0}\rangle.
\]

The structure of this paper is as follows. In Section~\ref{sec:bhm} we obtain some essential identities for the characteristic polynomials of finite banded Hessenberg matrices. In Section~\ref{sec:HP} we consider the operator $\mathcal{H}$ defined above and discuss the Hermite-Pad\'{e} approximation to the system of functions $(\phi_{1},\ldots,\phi_{p})$ in \eqref{Weylfunc:intro}. In Section~\ref{sec:tsop} we analyze the two-sided operator $\mathcal{M}$ and obtain some necessary estimates in the order of approximation of the Weyl functions \eqref{eq:idwj} by certain rational functions. Finally, in Section~\ref{sec:rcp} we prove Theorem~\ref{theo:main:intro}. We also investigate the joint probability distribution of the random vector $(\phi_{1}(z),\ldots,\phi_{p}(z))$, and its relation to the distributions $\mu_{k}$, $0\leq k\leq p$, and the function $\mathbb{E}(W(z))$.

\section{Banded lower Hessenberg matrices}\label{sec:bhm}

Throughout the paper we use the following notation. If $A$ is a square matrix, $A^{[r]}$ denotes the submatrix of $A$ obtained by deleting the first $r$ rows and columns of $A$, and $A_{[r]}$ denotes the submatrix obtained by deleting the last $r$ rows and columns of $A$. The $(i,j)$-entry of $A$ is denoted $A(i,j)$.

\begin{lemma}\label{lem:linalg}
Let $B_{n}=(m_{i,j})_{1\leq i,j\leq n}$ be the $n\times n$ matrix with $p+2$ diagonals, with entries
\[
\begin{cases}
m_{j-1,j}=1, & 2\leq j\leq n,\\
m_{j+k,j}=b_{j}^{(k)}, & 0\leq k\leq p,\,\,\,1\leq j\leq n-k,\,\,\,b_{j}^{(k)}\in\mathbb{C},\\
m_{i,j}=0, & \mbox{otherwise},
\end{cases}
\]
that is,
\[
B_{n}=\begin{pmatrix}
b_{1}^{(0)} & 1 & & & 0 \\
\vdots & \ddots & \ddots \\
b_{1}^{(p)} & & \ddots & \ddots & \\
 & \ddots & & \ddots & 1 \\
0 &  & b_{n-p}^{(p)} & \cdots & b_{n}^{(0)}
\end{pmatrix}.
\]
Let
\begin{equation}\label{def:charpol}
\begin{aligned}
Q_{n}(z) & :=\det(z I_{n}-B_{n}),\\
Q_{j}^{+}(z) & :=\det((z I_{n}-B_{n})^{[n-j]}),\quad 0\leq j\leq n-1,\\
Q_{j}^{-}(z) & :=\det((z I_{n}-B_{n})_{[n-j]}),\quad 0\leq j\leq n-1,
\end{aligned}
\end{equation}
where $Q_{0}^{\pm}\equiv 1$. We also define $Q_{\ell}^{\pm}\equiv 0$ for $\ell\leq -1$. 

We have
\begin{equation}\label{equ:1}
Q_{n}'(z)=\sum_{j=1}^{n} Q_{n-j}^{+}(z)\,Q_{j-1}^{-}(z),
\end{equation}
and for each $1\leq j\leq n$, 
\begin{equation}\label{equ:2}
Q_{n}(z)=(z-b^{(0)}_{j})\, Q_{n-j}^{+}(z)\,Q_{j-1}^{-}(z)-\sum_{\ell=1}^{p}\sum_{k=0}^{\ell} b_{j-k}^{(\ell)}\,Q_{n-j+k-\ell}^{+}(z)\,Q_{j-k-1}^{-}(z).
\end{equation}
If $z$ is not an eigenvalue of $B_{n}$, then for every $1\leq j\leq n$, 
\begin{align}
(z I_{n}-B_{n})^{-1}(j,j) & =\frac{Q_{n-j}^{+}(z)\,Q_{j-1}^{-}(z)}{Q_{n}(z)},\label{eq:traceresol:1}\\
(z I_{n}-B_{n})^{-1}(1,j) & =\frac{Q_{n-j}^{+}(z)}{Q_{n}(z)}.\label{eq:frapprox}
\end{align}
\end{lemma}
\begin{proof}
By Jacobi's derivative formula, if $A(z)$ is a square differentiable matrix, then
\[
\frac{d}{dz}\det(A(z))=\tr\left(\mathrm{adj}(A(z)) A'(z)\right),
\]
where $\mathrm{adj}(A(z))$ is the adjugate of $A(z)$. If we apply this formula to $A(z)=z I_{n}-B_{n}$, and note that the $(j,j)$-cofactor of $zI_{n}-B_{n}$ is $Q_{n-j}^{+}(z)\,Q_{j-1}^{-}(z)$, we obtain \eqref{equ:1} and \eqref{eq:traceresol:1}. 

Let $1\leq j\leq n$. If $r$ and $s$ are index sets, let $Q^{[r;s]}(z)$ denote the determinant of the submatrix of $z I_{n}-B_{n}$ obtained after deleting the rows with index in $r$ and the columns with index in $s$. By the adjugate formula for the inverse of a matrix,
\[
(z I_{n}-B_{n})^{-1}(1,j)=\frac{(-1)^{j+1} Q^{[j;1]}(z)}{Q_{n}(z)}=\frac{(-1)^{j+1}(-1)^{j-1} Q_{n-j}^{+}(z)}{Q_{n}(z)}=\frac{Q_{n-j}^{+}(z)}{Q_{n}(z)}
\]
which is \eqref{eq:frapprox}.

Expanding the determinant $Q_{n}(z)=\det(zI_{n}-B_{n})$ along row $j$, we obtain
\[
Q_{n}(z)=(z-b^{(0)}_{j})\,Q_{n-j}^{+}(z)\,Q_{j-1}^{-}(z)+Q^{[j;j+1]}(z)
-\sum_{\ell=1}^{p}(-1)^{\ell}\,b_{j-\ell}^{(\ell)}\,Q^{[j;j-\ell]}(z),
\]
understanding $Q^{[j;m]}\equiv 0$ if $m\leq 0$, and $Q^{[n;n+1]}\equiv 0$. It is easy to see that $Q^{[j;j-\ell]}(z)=(-1)^{\ell}\,Q_{n-j}^{+}(z)\,Q_{j-1-\ell}^{-}(z)$. So we have
\begin{equation}\label{equ:3}
Q_{n}(z)=(z-b^{(0)}_{j})\,Q_{n-j}^{+}(z)\,Q_{j-1}^{-}(z)+Q^{[j;j+1]}(z)
-\sum_{\ell=1}^{p}\,b_{j-\ell}^{(\ell)}\,Q_{n-j}^{+}(z)\,Q_{j-1-\ell}^{-}(z).
\end{equation}
Note that if $j=n$, then the proof of \eqref{equ:2} is complete since in this case it reduces to \eqref{equ:3}. 

Assume that $j\leq n-1$. To finish the proof of \eqref{equ:2}, we need to show that
\begin{equation}\label{eq:idQjjpo}
Q^{[j:j+1]}(z)=-\sum_{\ell=1}^{p}\sum_{k=0}^{\ell-1} b_{j-k}^{(\ell)}\,Q_{n-j+k-\ell}^{+}(z)\,Q_{j-k-1}^{-}(z).
\end{equation}
This formula involves the coefficients in the triangular array
\begin{equation}\label{triangarray}
\begin{array}{cccccc}
-b_{j-p+1}^{(p)} & -b_{j-p+2}^{(p-1)} & \cdots & -b_{j-2}^{(3)} & -b_{j-1}^{(2)} & -b_{j}^{(1)} \\[0.5em]
 & -b_{j-p+2}^{(p)} & \cdots & -b_{j-2}^{(4)} & -b_{j-1}^{(3)} & -b_{j}^{(2)} \\
  & & \ddots & \vdots & \vdots & \vdots \\
  & & & -b_{j-2}^{(p)} & -b_{j-1}^{(p-1)} & -b_{j}^{(p-2)} \\[0.5em]
  & & & & -b_{j-1}^{(p)} & -b_{j}^{(p-1)} \\[0.5em]
  & & & & & -b_{j}^{(p)}
\end{array}
\end{equation}
if $p\leq j\leq n-p$, but the array of coefficients in the formula will be smaller if $j$ is not in the indicated range. For example, if $j=1$, then the array reduces to a single column with $p$ coefficients, and if $j=n-1$, the array reduces to a single row with $p$ coefficients. If we expand the determinant $Q^{[j;j+1]}$ along the row that contains the first row of coefficients in \eqref{triangarray} (the $j$-th row in the determinant $Q^{[j;j+1]}$), then we obtain
\begin{align*}
Q^{[j;j+1]}(z) & =\sum_{\ell=1}^{p}(-b_{j-\ell+1}^{(\ell)})\,(-1)^{\ell-1}\, Q^{[j,j+1;j-\ell+1,j+1]}(z)+Q^{[j,j+1;j+1,j+2]}(z)\\
 & =-\sum_{\ell=1}^{p} b_{j-\ell+1}^{(\ell)}\,Q_{n-j-1}^{+}(z)\,Q_{j-\ell}^{-}(z)+Q^{[j,j+1;j+1,j+2]}(z),
\end{align*}
using $Q^{[j,j+1;j-\ell+1,j+1]}=(-1)^{\ell-1}\,Q^{+}_{n-j-1}\,Q_{j-\ell}^{-}$. If we now expand the determinant $Q^{[j,j+1;j+1,j+2]}$ along its $j$-th row, we will obtain an identity similar to the one for $Q^{[j;j+1]}$, but involving the coefficients in the second row in \eqref{triangarray}. After a repetition of this procedure $p$ times, we arrive at \eqref{eq:idQjjpo}.
\end{proof}

\section{One-sided operators and Hermite-Pad\'{e} approximation}\label{sec:HP}

We begin our discussion in this section with a collection of $p+1$ bounded deterministic sequences of complex numbers $(a_{n}^{(k)})_{n=1}^{\infty}$, $0\leq k\leq p$. With these sequences we construct the infinite matrix $H=(h_{i,j})_{i,j=1}^{\infty}$ with entries
\begin{equation}\label{def:osH}
\begin{cases}
h_{j-1,j}=1, & j\geq 2,\\
h_{j+k,j}=a_{j}^{(k)}, & 0\leq k\leq p, \quad j\geq 1,\\
h_{i,j}=0, & \mbox{otherwise}.
\end{cases}
\end{equation}

So the main diagonal of $H$ is formed by the sequence $(a_{n}^{(0)})$, the $k$-th subdiagonal, $1\leq k\leq p$, is formed by the sequence $(a^{(k)}_{n})$, the entries in the first superdiagonal are all equal to $1$, and the remaining entries are $0$. We denote by $H_{n}$ the principal $n\times n$ truncation of $H$:
\[
H_{n}=\begin{pmatrix}
a_{1}^{(0)} & 1 & & & 0 \\
\vdots & \ddots & \ddots \\
a_{1}^{(p)} & & \ddots & \ddots & \\
 & \ddots & & \ddots & 1 \\
0 &  & a_{n-p}^{(p)} & \cdots & a_{n}^{(0)}
\end{pmatrix}.
\] 

Let $\{e_{n}\}_{n=1}^{\infty}$ be the standard basis in $\ell^{2}(\mathbb{N})$. Consider the bounded operator $\mathcal{H}$ on $\ell^{2}(\mathbb{N})$ whose matrix representation in the standard basis is the matrix $H$, i.e., the operator defined by
\begin{equation}\label{def:op}
\begin{cases}
\mathcal{H} e_{1}=\sum_{k=0}^{p} a_{1}^{(k)} e_{k+1},\\
\mathcal{H} e_{n}=e_{n-1}+\sum_{k=0}^{p} a_{n}^{(k)} e_{n+k}, & n\geq 2. 
\end{cases}
\end{equation}
The boundedness of the operator follows from the boundedness of the diagonal sequences $(a_{n}^{(k)})$. Indeed, if $C>0$ is an upper bound for all $|a_{n}^{(k)}|$, it is easy to obtain the estimate $\|\mathcal{H}\|\leq (C+1)(p+2)$.

Let $(z I-\mathcal{H})^{-1}$ be the resolvent operator, and let 
\begin{equation}\label{def:resolvphi}
\phi_{j}(z):=\langle(z I-\mathcal{H})^{-1} e_{j}, e_{1}\rangle,\qquad 1\leq j\leq p.
\end{equation}
These functions are analytic in the complement of the spectrum of $\mathcal{H}$. The Laurent series at infinity of $\phi_{j}(z)$ is 
\begin{equation}\label{eq:Laurserphij}
\phi_{j}(z)=\sum_{n=0}^{\infty}\frac{\langle\mathcal{H}^{n} e_{j},e_{1}\rangle}{z^{n+1}}
=\frac{1}{z^{j}}+O\left(\frac{1}{z^{j+1}}\right),
\end{equation}
which converges absolutely for $|z|>\|\mathcal{H}\|$. The second equality follows easily from \eqref{def:op}, since $\langle\mathcal{H}^{j-1} e_{j},e_{1}\rangle=1$ and $\langle\mathcal{H}^{n} e_{j},e_{1}\rangle=0$ for $0\leq n\leq j-2$.

We also consider the difference equation of order $p+1$
\begin{equation}\label{def:diffeq}
z y_{n}=y_{n+1}+a^{(0)}_{n+1} y_{n}+a_{n}^{(1)} y_{n-1}+\cdots+a_{n-p+1}^{(p)} y_{n-p},\qquad n\geq p.
\end{equation}
A basis for the space of all solutions $(y_{n})_{n=0}^{\infty}$ of \eqref{def:diffeq} is formed by the following $p+1$ polynomial sequences:
\begin{equation}\label{polsol}
\begin{aligned}
q_{n}(z) & :=\det(z I_{n}-H_{n}), \qquad n\geq 0,\\
q_{n,j}(z) & :=\det((z I_{n}-H_{n})^{[j]}),\qquad 1\leq j\leq p,\quad n\geq 0.
\end{aligned}
\end{equation}
Recall that $(z I_{n}-H_{n})^{[j]}$ is the submatrix of $z I_{n}-H_{n}$ obtained 
after deleting the first $j$ rows and columns. With this notation, we understand that $q_{0}, q_{j,j}\equiv 1$, and $q_{n,j}\equiv 0$ for $n<j$. These conditions show that the sequences in \eqref{polsol} are linearly independent. To see that they are indeed solutions of \eqref{def:diffeq}, expand the determinant $\det((z I_{n+1}-H_{n+1})^{[j]})$ along its last row. Note that 
$q_{n}$ is of degree $n$, and $q_{n,j}$ is of degree $n-j$ for $n\geq j$. The following result is fundamental for our analysis.

\begin{theorem}[Kalyagin \cite{Kal}, see also \cite{AptKal}]\label{theo:Kal}
For each $n\geq 0$, the vector of rational functions
\[
\left(\frac{q_{n,1}}{q_{n}}, \frac{q_{n,2}}{q_{n}},\ldots,\frac{q_{n,p}}{q_{n}}\right)
\]
is an Hermite-Pad\'{e} approximant at infinity for the system of resolvent functions $(\phi_{1}, \phi_{2},\ldots,\phi_{p})$, with respect to the multi-index
\begin{equation}\label{eq:HPindex}
(n_{1},n_{2},\ldots,n_{p})=(\underbrace{k+1,k+1,\ldots,k+1}_{s},k,\ldots,k),
\end{equation}
where $n=kp+s$ is the decomposition of $n$ modulo $p$. This means
\begin{equation}\label{eq:HPprop}
q_{n}(z)\, \phi_{j}(z)-q_{n,j}(z)=O\left(\frac{1}{z^{n_{j}+1}}\right),\quad z\rightarrow\infty,
\end{equation}
for each $1\leq j\leq p$. 
\end{theorem} 

The $j$th component of \eqref{eq:HPindex} is $n_{j}=\lfloor (n-j)/p\rfloor+1$, where $\lfloor\cdot\rfloor$ is the floor function. Note that by \eqref{eq:frapprox},
\begin{equation}\label{eq:resolvqnqnj}
\frac{q_{n,j}(z)}{q_{n}(z)}=(z I_{n}-H_{n})^{-1}(1,j),\qquad 1\leq j\leq p,
\end{equation}
see the analogy between this formula and \eqref{def:resolvphi}.

If we eliminate the first row and the first column of the infinite matrix $H$ in \eqref{def:osH}, we obtain an infinite matrix $H_{1}$ with corresponding operator $\mathcal{H}_{1}$ on $\ell^{2}(\mathbb{N})$ given by
\[
\begin{cases}
\mathcal{H}_{1} e_{1}=\sum_{k=0}^{p} a_{2}^{(k)} e_{k+1},\\
\mathcal{H}_{1} e_{n}=e_{n-1}+\sum_{k=0}^{p} a_{n+1}^{(k)} e_{n+k}, & n\geq 2. 
\end{cases}
\]
The associated resolvent functions are
\begin{equation}\label{def:resolvphi1j}
\phi_{1,j}(z):=\langle(zI-\mathcal{H}_{1})^{-1} e_{j},e_{1}\rangle,\qquad 1\leq j\leq p.
\end{equation}
It is clear that the functions $\phi_{1,j}$ are also analytic in $\{z\in\mathbb{C}: |z|>\|\mathcal{H}\|\}$.

\begin{lemma}\label{lem:relfuncphi}
The following relations hold for every $|z|>\|\mathcal{H}\|$,
\begin{align}
\phi_{1}(z) & =\frac{1}{z-a_{1}^{(0)}-\sum_{k=1}^{p} a_{1}^{(k)} \phi_{1,k}(z)},\label{rel:phis:1}\\
\phi_{j}(z) & =\frac{\phi_{1,j-1}(z)}{z-a_{1}^{(0)}-\sum_{k=1}^{p} a_{1}^{(k)} \phi_{1,k}(z)},\qquad 2\leq j\leq p.\label{rel:phis:2}
\end{align}
\end{lemma}
\begin{proof}
If we expand the determinant $q_{n}(z)=\det(z I_{n}-H_{n})$ along its first column, we obtain
\begin{equation}\label{eq:expqnqnj}
q_{n}(z)=(z-a_{1}^{(0)})\,q_{n,1}(z)-\sum_{j=1}^{p} a_{1}^{(j)} q_{n,j+1}(z),
\end{equation}
where $q_{n,p+1}(z):=\det((zI_{n}-H_{n})^{[p+1]})$, and the other polynomials are defined in \eqref{polsol}.

Recall that we use the following notation. For the Laurent series $f(z)=\sum_{k\in\mathbb{Z}} c_{k}\,z^{-k}$, we write $c_{k}=[f]_{k}$. We first prove \eqref{rel:phis:1} by showing that both sides of
\begin{equation}\label{eq:Laur}
(z-a_{1}^{(0)})\,\phi_{1}(z)-\sum_{j=1}^{p} a_{1}^{(j)}\,\phi_{1,j}(z)\,\phi_{1}(z)=1
\end{equation}
have the same Laurent series at infinity. 

First, observe that the vector
\[
\left(\frac{q_{n,2}}{q_{n,1}},\frac{q_{n,3}}{q_{n,1}},\ldots,\frac{q_{n,p+1}}{q_{n,1}}\right)
\]
is an Hermite-Pad\'{e} approximant of the system of functions $(\phi_{1,1},\ldots,\phi_{1,p})$. By Theorem~\ref{theo:Kal}, we can write
\begin{align*}
q_{n}(z)\,\phi_{j}(z)-q_{n,j}(z) & =\varepsilon_{n,j}(z)\\
q_{n,1}(z)\,\phi_{1,j}(z)-q_{n,j+1}(z) & =\rho_{n,j}(z)
\end{align*} 
where for each $k\geq 0$ fixed, $[\varepsilon_{n,j}]_{k}, [\rho_{n,j}]_{k}$ are zero for all $n$ large enough. Therefore,
\begin{equation}\label{eq:HPprop:2}
\begin{aligned}
& [\phi_{j}]_{k}=\left[\frac{q_{n,j}}{q_{n}}+\frac{\varepsilon_{n,j}}{q_{n}}\right]_{k}=\left[\frac{q_{n,j}}{q_{n}}\right]_{k}\\
& [\phi_{1,j}]_{k}=\left[\frac{q_{n,j+1}}{q_{n,1}}+\frac{\rho_{n,j}}{q_{n,1}}\right]_{k}=\left[\frac{q_{n,j+1}}{q_{n,1}}\right]_{k} 
\end{aligned}
\end{equation}
for each fixed $k\geq 0$ and all $n$ large enough.

Dividing \eqref{eq:expqnqnj} by $q_{n}$, we have
\[
(z-a_{1}^{(0)})\,\frac{q_{n,1}(z)}{q_{n}(z)}-\sum_{j=1}^{p}a_{1}^{(j)}\,\frac{q_{n,j+1}(z)}{q_{n,1}(z)}\frac{q_{n,1}(z)}{q_{n}(z)}=1,
\]
which together with \eqref{eq:HPprop:2} implies \eqref{eq:Laur}.

The relation \eqref{rel:phis:2} is then equivalent to $\phi_{j}=\phi_{1,j-1}\,\phi_{1}$, $2\leq j\leq p$.  Writing $\frac{q_{n,j}}{q_{n}}=\frac{q_{n,j}}{q_{n,1}}\frac{q_{n,1}}{q_{n}}$, we have, for all $n$ large enough,
\[
[\phi_{j}]_{k}=\left[\frac{q_{n,j}}{q_{n}}\right]_{k}=\left[\frac{q_{n,j}}{q_{n,1}}\frac{q_{n,1}}{q_{n}}\right]_{k}=[\phi_{1,j-1} \phi_{1}]_{k},
\]
and so $\phi_{j}=\phi_{1,j-1}\,\phi_{1}$.
\end{proof}

For $n\in\mathbb{Z}_{\geq 0}$, we define the set
\begin{equation}\label{def:Cn}
C(n):=\{\mathbf{k}=(k_{0},\ldots,k_{p})\in\mathbb{Z}^{p+1}_{\geq 0}: k_{0}+\cdots+k_{p}=n\}.
\end{equation}
Given $\mathbf{k}=(k_{0},\ldots,k_{p})\in C(n)$, we will use the notation
\[
\binom{n}{\mathbf{k}}:=\frac{n!}{k_{0}!\,k_1!\,\cdots\, k_{p}!}
\]
and the convention
\[
\binom{n}{-1}=\delta_{n,-1}=\begin{cases}
1\quad \mbox{if}\,\,n=-1,\\
0\quad \mbox{if}\,\,n\neq -1.
\end{cases}
\]
The $j$-th component of $\mathbf{k}$ will be denoted $\mathbf{k}(j)$. In addition to the functions in \eqref{def:resolvphi1j}, we define $\phi_{1,0}\equiv 1$.

For our analysis later, we need the following simple consequence of Lemma~\ref{lem:relfuncphi}.

\begin{corollary}
For a vector $(r_{1},\ldots,r_{p})\in\mathbb{Z}_{\geq 0}^{p}$, we have for all $z$ large enough the relation
\begin{equation}\label{eq:phiphi1}
\prod_{j=1}^{p}\phi_{j}(z)^{r_{j}}=\sum_{n=0}^{\infty}\sum_{\mathbf{k}\in C(n)}\frac{1}{z^{n+r}}
\binom{n+r-1}{r-1}\binom{n}{\mathbf{k}}\Big(\prod_{j=0}^{p} (a_{1}^{(j)})^{\mathbf{k}(j)}\Big)\prod_{j=1}^{p}\phi_{1,j}(z)^{\mathbf{k}(j)+r_{j+1}}
\end{equation}
where $r_{p+1}=0$ and $r=\sum_{j=1}^{p} r_{j}$.
\end{corollary}
\begin{proof}
According to \eqref{rel:phis:1}--\eqref{rel:phis:2},
\[
\prod_{j=1}^{p}\phi_{j}(z)^{r_{j}}=\frac{\prod_{j=2}^{p}\phi_{1,j-1}(z)^{r_{j}}}{(z-\sum_{k=0}^{p} a_{1}^{(k)}\,\phi_{1,k}(z))^{\sum_{j=1}^{p}r_{j}}}.
\] 
Let $r=\sum_{j=1}^{p} r_{j}$. Using the identity $(z-\rho)^{-r}=\sum_{n=0}^{\infty}\binom{n+r-1}{r-1}\,\rho^{n}\,z^{-n-r}$, we obtain
\begin{align*}
\frac{1}{(z-\sum_{k=0}^{p} a_{1}^{(k)}\,\phi_{1,k}(z))^{r}}& =\sum_{n=0}^{\infty}\binom{n+r-1}{r-1}\frac{(\sum_{k=0}^{p}a_{1}^{(k)}\,\phi_{1,k}(z))^{n}}{z^{n+r}}\\
& =\sum_{n=0}^{\infty}\sum_{\mathbf{k}\in C(n)}\frac{1}{z^{n+r}}\binom{n+r-1}{r-1}\binom{n}{\mathbf{k}}\prod_{j=0}^{p}(a^{(j)}_{1}\,\phi_{1,j}(z))^{\mathbf{k}(j)}
\end{align*}
and the result follows. 
\end{proof}

\section{Two-sided operator}\label{sec:tsop}

In this section we consider banded Hessenberg operators on $\ell^{2}(\mathbb{Z})$, and obtain a connection formula between Weyl functions of such operators and certain restriction operators on $\ell^{2}(\mathbb{N})$. In the case of Jacobi operators (the case $p=1$), relations between the spectral properties of one-sided and two-sided operators have been extensively investigated. We mention in this area the pioneering work of Nikishin \cite{Nik2}. Other important and more recent works are for example \cite{MasRep,DKS,DaiIsmWang}, see also \cite{Simon} and references therein.

Consider $p+1$ deterministic bounded sequences of complex numbers $(a_{n}^{(k)})_{n\in\mathbb{Z}}$, $0\leq k\leq p$. Let $\{\hat{e}_{n}\}_{n\in\mathbb{Z}}$ be the standard basis in the space $\ell^{2}(\mathbb{Z})$, and let $\mathcal{M}$ be the bounded operator on $\ell^{2}(\mathbb{Z})$ that acts on the standard basis vectors as follows: 
\begin{equation}\label{def:opB}
\mathcal{M}\hat{e}_{n}=\hat{e}_{n-1}+\sum_{k=0}^{p} a_{n}^{(k)} \hat{e}_{n+k},\qquad n\in\mathbb{Z}.
\end{equation}
The matrix representation of $\mathcal{M}$ in the basis $\{\hat{e}_{n}\}_{n\in\mathbb{Z}}$ is the bi-infinite matrix $M=(m_{i,j})_{i,j\in\mathbb{Z}}$ with entries
\begin{equation}\label{def:matrixB}
\begin{cases}
m_{j-1,j}=1, & j\in\mathbb{Z},\\
m_{j+k,j}=a_{j}^{(k)}, & 0\leq k\leq p, \quad j\in\mathbb{Z},\\
m_{i,j}=0, & \mbox{otherwise}.
\end{cases}
\end{equation}

Let $r\in\mathbb{Z}$ be fixed. If we focus on the entries $m_{i,j}$ of the matrix $M$ with $i,j\geq r+1$, the resulting submatrix is associated with the operator $\mathcal{H}_{r}^{+}$ on $\ell^{2}(\mathbb{N})$ defined by
\begin{equation}\label{def:opHp}
\begin{cases}
\mathcal{H}^{+}_{r} e_{1}=\sum_{k=0}^{p} a_{r+1}^{(k)}\,e_{k+1},\\
\mathcal{H}^{+}_{r} e_{n}=e_{n-1}+\sum_{k=0}^{p} a_{n+r}^{(k)}\,e_{n+k}, & n\geq 2. 
\end{cases}
\end{equation}
So the matrix representation of $\mathcal{H}_{r}^{+}$ in the basis $\{e_{n}\}_{n=1}^{\infty}$ is
\[
\begin{pmatrix}
a_{r+1}^{(0)} & 1 & & & \\
\vdots & a_{r+2}^{(0)} & \ddots \\
a_{r+1}^{(p)} & \vdots & \ddots & \\
 & a_{r+2}^{(p)} & & \\
 &  & \ddots & \\
 & & &   
\end{pmatrix}.
\]
Similarly, for $r\in\mathbb{Z}$, if we restrict ourselves to the entries $m_{i,j}$ of the matrix $M$ with $i,j\leq r-1$, the resulting matrix is linked to the operator $\mathcal{H}_{r}^{-}$ on $\ell^{2}(\mathbb{N})$ defined by
\begin{equation}\label{def:opHm}
\begin{cases}
\mathcal{H}^{-}_{r} e_{1}=\sum_{k=0}^{p} a_{r-1-k}^{(k)}\,e_{k+1},\\
\mathcal{H}^{-}_{r} e_{n}=e_{n-1}+\sum_{k=0}^{p} a_{r-n-k}^{(k)}\,e_{n+k}, & n\geq 2. 
\end{cases}
\end{equation}
So the matrix representation of $\mathcal{H}_{r}^{-}$ in the basis $\{e_{n}\}_{n=1}^{\infty}$ is
\[
\begin{pmatrix}
a_{r-1}^{(0)} & 1 & & & \\
\vdots & a_{r-2}^{(0)} & \ddots \\
a_{r-p-1}^{(p)} & \vdots & \ddots & \\
 & a_{r-p-2}^{(p)} & & \\
 &  & \ddots & \\
 & & &   
\end{pmatrix}.
\]
The corresponding resolvent functions are
\begin{equation}\label{def:resolfunc}
\phi_{r,k}^{\pm}(z):=\langle(zI-\mathcal{H}_{r}^{\pm})^{-1} e_{k},e_{1}\rangle,\quad 1\leq k\leq p.
\end{equation}
We also set, by definition, $\phi_{r,0}^{\pm}(z)\equiv 1$.

We discuss now how certain resolvent functions associated with the operator $\mathcal{M}$ defined in \eqref{def:opB} can be approximated by certain rational functions built with characteristic polynomials.

For $n\geq 0$, let $M_{2n+1}$ be the submatrix of $M$ of size $(2n+1)\times (2n+1)$ with entries $m_{i,j}$, $-n\leq i,j\leq n$. The entry at the center of this matrix is $m_{0,0}=a^{(0)}_{0}$. In accordance with the notation used in \eqref{lem:linalg}, we define the polynomials
\begin{align*}
Q_{2n+1}(z) & :=\det(z I_{2n+1}-M_{2n+1}),\\
Q_{\ell}^{+}(z) & :=\det((z I_{2n+1}-M_{2n+1})^{[2n+1-\ell]}),\quad 1\leq \ell\leq 2n,\\
Q_{\ell}^{-}(z) & :=\det((z I_{2n+1}-M_{2n+1})_{[2n+1-\ell]}),\quad 1\leq \ell\leq 2n,
\end{align*} 
so note that $Q_{\ell}^{\pm}$ are polynomials of degree $\ell$. As before, we set $Q_{0}^{\pm}\equiv 1$, and $Q_{\ell}^{\pm}\equiv 0$ for $\ell\leq -1$. Applying Lemma~\ref{lem:linalg} to the matrix $M_{2n+1}$, we obtain
\[
Q_{2n+1}'(z)=\sum_{j=-n}^{n} Q_{n-j}^{+}(z)\,Q_{n+j}^{-}(z)
\]
and for each $-n\leq j\leq n$,
\[
Q_{2n+1}(z)=(z-a_{j}^{(0)})\,Q_{n-j}^{+}(z)\,Q_{n+j}^{-}(z)
-\sum_{\ell=1}^{p}\sum_{k=0}^{\ell} a_{j-k}^{(\ell)}\,Q_{n-j+k-\ell}^{+}(z)\,Q_{n+j-k}^{-}(z).
\]
Hence,
\begin{align}
\frac{Q_{2n+1}'(z)}{Q_{2n+1}(z)} & =\sum_{j=-n}^{n}\frac{Q_{n-j}^{+}(z)\,Q_{n+j}^{-}(z)}{Q_{2n+1}(z)}\notag\\
& =\sum_{j=-n}^{n}\frac{1}{z-a^{(0)}_{j}-\sum_{\ell=1}^{p}\sum_{k=0}^{\ell} a_{j-k}^{(\ell)}\,\frac{Q_{n-j+k-\ell}^{+}(z)}{Q_{n-j}^{+}(z)}\,\frac{Q_{n+j-k}^{-}(z)}{Q_{n+j}^{-}(z)}}.\label{eq:ratioQs}
\end{align}
Let us analyze the quotient $Q_{n-j+k-\ell}^{+}(z)/Q_{n-j}^{+}(z)$. The denominator is the characteristic polynomial of the square matrix
\begin{equation}\label{eq:matcharpol}
\begin{pmatrix}
a_{j+1}^{(0)} & 1 & & & 0 \\
\vdots & \ddots & \ddots \\
a_{j+1}^{(p)} & & \ddots & \ddots & \\
 & \ddots & & \ddots & 1 \\
0 &  & a_{n-p}^{(p)} & \cdots & a_{n}^{(0)}
\end{pmatrix}
\end{equation}
and the numerator is the characteristic polynomial of the submatrix of \eqref{eq:matcharpol} obtained after deleting the first $\ell-k$ rows and columns. Therefore, according to Theorem~\ref{theo:Kal}, we can view the function $Q_{n-j+k-\ell}^{+}(z)/Q_{n-j}^{+}(z)$ as an approximation of the resolvent function $\phi_{j,\ell-k}^{+}(z)$ associated with the operator $\mathcal{H}_{j}^{+}$. In virtue of \eqref{eq:HPprop}, we have
\begin{equation}\label{eq:estinf:1}
\phi_{j,\ell-k}^{+}(z)-\frac{Q_{n-j+k-\ell}^{+}(z)}{Q_{n-j}^{+}(z)}=O\left(\frac{1}{z^{n-j+2+\lfloor(n-j+k-\ell)/p\rfloor}}\right),\quad z\rightarrow\infty.
\end{equation}
Similarly, we can view $Q_{n+j-k}^{-}(z)/Q_{n+j}^{-}(z)$ as an approximation of the resolvent function $\phi_{j,k}^{-}(z)$ associated with the operator $\mathcal{H}_{j}^{-}$, and according to \eqref{eq:HPprop} we have
\begin{equation}\label{eq:estinf:2}
\phi_{j,k}^{-}(z)-\frac{Q_{n+j-k}^{-}(z)}{Q_{n+j}^{-}(z)}=O\left(\frac{1}{z^{n+j+2+\lfloor(n+j-k)/p\rfloor}}\right),\quad z\rightarrow\infty.
\end{equation}

Consider now the function
\begin{equation}\label{def:funcwj}
w_{j}(z):=\frac{1}{z-a^{(0)}_{j}-\sum_{\ell=1}^{p}\sum_{k=0}^{\ell} a_{j-k}^{(\ell)}\,\phi_{j,\ell-k}^{+}(z)\,\phi_{j,k}^{-}(z)},
\end{equation}
which we will compare with $\frac{Q_{n-j}^{+}(z)Q_{n+j}^{-}(z)}{Q_{2n+1}(z)}$, see \eqref{eq:ratioQs}. Applying \eqref{eq:estinf:1} and \eqref{eq:estinf:2}, we obtain for each $-n\leq j\leq n$ the estimate
\begin{gather}
\frac{Q_{n-j}^{+}(z)Q_{n+j}^{-}(z)}{Q_{2n+1}(z)}-w_{j}(z)\notag\\
=\frac{\sum_{\ell=1}^{p}\sum_{k=0}^{\ell} a_{j-k}^{(\ell)}
\left(\frac{Q_{n-j+k-\ell}^{+}(z)}{Q_{n-j}^{+}(z)}
\frac{Q_{n+j-k}^{-}(z)}{Q_{n+j}^{-}(z)}-\phi_{j,\ell-k}^{+}(z)\phi_{j,k}^{-}(z)\right)}{\left(z-a^{(0)}_{j}-\sum_{\ell}\sum_{k}a_{j-k}^{(\ell)}\,\frac{Q_{n-j+k-\ell}^{+}(z)Q_{n+j-k}^{-}(z)}{Q_{n-j}^{+}(z)Q_{n+j}^{-}(z)}\right)\left(z-a^{(0)}_{j}-\sum_{\ell}\sum_{k}a_{j-k}^{(\ell)}\,\phi_{j,\ell-k}^{+}(z)\phi_{j,k}^{-}(z)\right)}\notag\\
=O\left(\frac{1}{z^{n-|j|+3+\lfloor(n-|j|)/p\rfloor}}\right).\label{eq:estinf:3}
\end{gather}
Indeed, the denominator in the second expression in \eqref{eq:estinf:3} gives the contribution $O(z^{-2})$, and we have
\begin{align*}
\frac{Q_{n-j+k-\ell}^{+}}{Q_{n-j}^{+}}\frac{Q_{n+j-k}^{-}}{Q_{n+j}^{-}}-\phi_{j,\ell-k}^{+}\,\phi_{j,k}^{-} & =\frac{Q_{n+j-k}^{-}}{Q_{n+j}^{-}}\left(\frac{Q_{n-j+k-\ell}^{+}}{Q_{n-j}^{+}}-\phi_{j,\ell-k}^{+}\right)+\phi_{j,\ell-k}^{+}\left(\frac{Q_{n+j-k}^{-}}{Q_{n+j}^{-}}-\phi_{j,k}^{-}\right)\\
& =O\left(\frac{1}{z^{k+n-j+2+\lfloor(n-j+k-\ell)/p\rfloor}}\right)
+O\left(\frac{1}{z^{\ell-k+n+j+2+\lfloor(n+j-k)/p\rfloor}}\right)\\
& =O\left(\frac{1}{z^{n-|j|+1+\lfloor(n-|j|)/p\rfloor}}\right)
\end{align*} 
and \eqref{eq:estinf:3} follows.

If $M_{2n}$ is now the submatrix of $M$ of size $2n \times 2n$ with entries $m_{i,j}$, $-n+1\leq i,j\leq n$, and we let
\begin{align*}
Q_{2n}(z) & :=\det(z I_{2n}-M_{2n}),\\
Q_{\ell}^{+}(z) & :=\det((z I_{2n}-M_{2n})^{[2n-\ell]}),\quad 1\leq \ell\leq 2n,\\
Q_{\ell}^{-}(z) & :=\det((z I_{2n}-M_{2n})_{[2n-\ell]}),\quad 1\leq \ell\leq 2n,
\end{align*}
then similar computations yield
\begin{align*}
\frac{Q_{2n}'(z)}{Q_{2n}(z)} & =\sum_{j=-n+1}^{n}\frac{Q_{n-j}^{+}(z)\,Q_{n+j-1}^{-}(z)}{Q_{2n}(z)}\\
& =\sum_{j=-n+1}^{n}\frac{1}{z-a^{(0)}_{j}-\sum_{\ell=1}^{p}\sum_{k=0}^{\ell} a_{j-k}^{(\ell)}\,\frac{Q_{n-j+k-\ell}^{+}(z)}{Q_{n-j}^{+}(z)}\,\frac{Q_{n+j-k-1}^{-}(z)}{Q_{n+j-1}^{-}(z)}}
\end{align*}
and for every $-n+1\leq j\leq n$, we have
\begin{equation}\label{eq:estinf:4}
\frac{Q_{n-j}^{+}(z)\,Q_{n+j-1}^{-}(z)}{Q_{2n}(z)}-w_{j}(z)=O\left(\frac{1}{z^{n-|j|+1+\lfloor(n-|j|)/p\rfloor}}\right).
\end{equation}

\begin{lemma}
For all $j\in\mathbb{Z}$,
\begin{equation}\label{eq:idwj}
w_{j}(z)=\langle(zI-\mathcal{M})^{-1}\hat{e}_{j},\hat{e}_{j}\rangle,\qquad |z|>\|\mathcal{M}\|.
\end{equation}
\end{lemma}
\begin{proof}
We check that both functions in \eqref{eq:idwj} have the same Laurent series at infinity. Recall that if $f(z)$ is a Laurent series, we denote by $[f]_{k}$ the coefficient multiplying $z^{-k}$ in the series. 

Let $j\in\mathbb{Z}$ be fixed. Consider the matrix $M_{2n+1}$ defined before in this section. If $n\geq |j|$, by \eqref{eq:traceresol:1} we have
\[
(z I_{2n+1}-M_{2n+1})^{-1}(j,j)=\frac{Q_{n-j}^{+}(z) Q_{n+j}^{-}(z)}{Q_{2n+1}(z)},
\]
hence
\[
\frac{Q_{n-j}^{+}(z) Q_{n+j}^{-}(z)}{Q_{2n+1}(z)}=\sum_{s=0}^{\infty}\frac{M_{2n+1}^{s}(j,j)}{z^{s+1}}.
\]
It follows from \eqref{eq:estinf:3} that
\[
\left[\frac{Q_{n-j}^{+} Q_{n+j}^{-}}{Q_{2n+1}}\right]_{k}=\left[w_{j}\right]_{k},\qquad 0\leq k\leq n-|j|+2+\lfloor(n-|j|)/p\rfloor.
\]
On the other hand, we have
\[
\langle(zI-\mathcal{M})^{-1}\hat{e}_{j},\hat{e}_{j}\rangle=\sum_{s=0}^{\infty}\frac{\langle\mathcal{M}^{s}\hat{e}_{j},\hat{e}_{j}\rangle}{z^{s+1}},\qquad |z|>\|\mathcal{M}\|.
\]
It is clear that for each fixed $s\geq 0$, 
\[
\langle\mathcal{M}^{s}\hat{e}_{j},\hat{e}_{j}\rangle=M_{2n+1}^{s}(j,j)=\left[\frac{Q_{n-j}^{+} Q_{n+j}^{-}}{Q_{2n+1}}\right]_{s+1}
\]
for all $n$ large enough. We conclude that $\langle\mathcal{M}^{s}\hat{e}_{j},\hat{e}_{j}\rangle=[w_{j}]_{s+1}$ for all $s\geq 0$, so \eqref{eq:idwj} follows.
\end{proof}

In what follows we need some definitions. Given an integer $r\in\mathbb{Z}_{\geq 0}$, let
\[
\widehat{C}(r):=\{\mathbf{k}=(k_{1},\ldots,k_{\frac{(p+1)(p+2)}{2}})\in\mathbb{Z}_{\geq 0}^{\frac{(p+1)(p+2)}{2}}: k_{1}+\cdots+k_{\frac{(p+1)(p+2)}{2}}=r\}.
\]
For $\mathbf{k}\in \widehat{C}(r)$, we will use the notation
\[
\binom{r}{\mathbf{k}}=\frac{r!}{k_{1}!\,k_{2}!\,\cdots\,k_{\frac{(p+1)(p+2)}{2}}!}.
\]
The $i$-th component of $\mathbf{k}$ will be denoted $\mathbf{k}(i)$. Given 
$\mathbf{k}\in \widehat{C}(r)$ and $1\leq j\leq p$, we define
\begin{equation}\label{def:alphabeta:1}
\begin{aligned}
\alpha(\mathbf{k},j) & :=\sum_{t=1}^{p-j+1}\mathbf{k}\Big(\frac{(j+t)(j+t+1)}{2}-j\Big),\\
\beta(\mathbf{k},j) & :=\sum_{t=1}^{p-j+1}\mathbf{k}\Big(\frac{(j+t-1)(j+t)}{2}+j+1\Big),
\end{aligned}
\end{equation}
and the vectors
\begin{equation}\label{def:alphabeta:2}
\begin{aligned}
\alpha(\mathbf{k}) & :=(\alpha(\mathbf{k},1),\alpha(\mathbf{k},2),\ldots,\alpha(\mathbf{k},p)), \\
\beta(\mathbf{k}) & :=(\beta(\mathbf{k},1),\beta(\mathbf{k},2),\ldots,\beta(\mathbf{k},p)).
\end{aligned}
\end{equation}

Let us also introduce the functions
\[
f_{k,\ell}:=a_{-k}^{(\ell)}\,\phi_{0,\ell-k}^{+}\,\phi_{0,k}^{-},\qquad 0\leq \ell\leq p,\quad 0\leq k\leq \ell.
\]
Note that there are a total of $(p+1)(p+2)/2$ such functions. Recall that by definition, $\phi_{0,0}^{+}(z)\equiv \phi_{0,0}^{-}(z)\equiv 1$.

\begin{lemma}
The following formula holds:
\begin{equation}\label{eq:w0}
w_{0}(z)=\sum_{r=0}^{\infty}\sum_{\mathbf{k}\in \widehat{C}(r)}\frac{1}{z^{r+1}}\binom{r}{\mathbf{k}}\,\prod_{0\leq s\leq \ell\leq p} (a^{(\ell)}_{-s})^{\mathbf{k}(\frac{\ell(\ell+1)}{2}+1+s)}\prod_{j=1}^{p}\left\{\phi_{0,j}^{+}(z)^{\alpha(\mathbf{k},j)}\,\phi_{0,j}^{-}(z)^{\beta(\mathbf{k},j)}\right\}.
\end{equation}
\end{lemma}
\begin{proof}
We have
\[
w_{0}(z)=\frac{1}{z-\sum_{0\leq k\leq \ell\leq p} f_{k,\ell}(z)}=\sum_{r=0}^{\infty}\frac{\left(\sum_{0\leq k\leq \ell\leq p} f_{k,\ell}(z)\right)^{r}}{z^{r+1}}.
\]
Let $g_{i}(z)$, $1\leq i\leq \frac{(p+1)(p+2)}{2}$, be a relabeling of the functions $f_{k,\ell}(z)$ as indicated below:
\begin{equation}\label{arrayfg}
\begin{array}{ccccc}
f_{0,0} & f_{0,1} & f_{0,2} & \cdots & f_{0,p} \\[0.3em]
 & f_{1,1} & f_{1,2} & \cdots & f_{1,p} \\[0.3em]
  & & f_{2,2} & \cdots & f_{2,p} \\[0.3em]
  & & & \ddots & \vdots\\
  & & & & f_{p,p}
\end{array}
\qquad
\begin{array}{ccccc}
g_{1} & g_{2} & g_{4} & \cdots & g_{\frac{p(p+1)}{2}+1} \\
 & g_{3} & g_{5} & \cdots & g_{\frac{p(p+1)}{2}+2} \\
  & & g_{6} & \cdots & g_{\frac{p(p+1)}{2}+3} \\
  & & & \ddots & \vdots\\
  & & & & g_{\frac{(p+1)(p+2)}{2}}
\end{array}
\end{equation}
that is, we make the identification $f_{k,\ell}=g_{\frac{\ell(\ell+1)}{2}+k+1}$, $0\leq k\leq \ell\leq p$. So
\[
\Bigg(\sum_{0\leq k\leq \ell\leq p} f_{k,\ell}(z)\Bigg)^{r}=\Bigg(\sum_{i=1}^{\frac{(p+1)(p+2)}{2}}g_{i}(z)\Bigg)^{r}=\sum_{\mathbf{k}\in \widehat{C}(r)}\binom{r}{\mathbf{k}}\prod_{i=1}^{\frac{(p+1)(p+2)}{2}}g_{i}(z)^{\mathbf{k}(i)}.
\]
Now we find the power of $\phi_{0,j}^{+}(z)$ that appears in the product $\prod_{i=1}^{\frac{(p+1)(p+2)}{2}}g_{i}(z)^{\mathbf{k}(i)}$. The triangular array \eqref{arrayfg} has $p+1$ diagonals. Observe that the function $\phi_{0,1}^{+}$ appears only in the first superdiagonal of the array, the function $\phi_{0,2}^{+}$ appears only in the second superdiagonal, and so on. For each $1\leq j\leq p$, the function $\phi_{0,j}^{+}$ appears only in the $j$-th superdiagonal of \eqref{arrayfg}, the one that contains the functions 
\[
g_{\frac{j(j+1)}{2}+1},\,\,\,g_{\frac{j(j+1)}{2}+1+(j+2)},\,\,\,g_{\frac{j(j+1)}{2}+1+(j+2)+(j+3)},\,\,\,\ldots
\]
Hence, in the product $\prod_{i=1}^{\frac{(p+1)(p+2)}{2}}g_{i}(z)^{\mathbf{k}(i)}$, the function $\phi_{0,j}^{+}(z)$ appears raised to the power
\[
\sum_{t=1}^{p-j+1}\mathbf{k}\Big(\frac{j(j+1)}{2}+1+(j+2)+\cdots+(j+t)\Big)=\sum_{t=1}^{p-j+1}\mathbf{k}\Big(\frac{(j+t)(j+t+1)}{2}-j\Big)=\alpha(\mathbf{k},j).
\]
Now, the function $\phi_{0,j}^{-}(z)$ appears only in row $j+1$ of the array \eqref{arrayfg}, so in the product $\prod_{i=1}^{\frac{(p+1)(p+2)}{2}}g_{i}(z)^{\mathbf{k}(i)}$ the function $\phi_{0,j}^{-}(z)$ appears raised to the power
\[
\sum_{t=1}^{p-j+1}\mathbf{k}\Big(\frac{(j+1)(j+2)}{2}+(j+1)+\cdots+(j+t-1)\Big)=\sum_{t=1}^{p-j+1}\mathbf{k}\Big(\frac{(j+t-1)(j+t)}{2}+j+1\Big)=\beta(\mathbf{k},j).
\]
It is readily seen that in the product $\prod_{i=1}^{\frac{(p+1)(p+2)}{2}}g_{i}(z)^{\mathbf{k}(i)}$, the coefficient $a_{-s}^{(\ell)}$ appears raised to the power $\mathbf{k}(\frac{\ell(\ell+1)}{2}+1+s)$.

In conclusion, we have
\[
\prod_{i=1}^{\frac{(p+1)(p+2)}{2}}g_{i}(z)^{\mathbf{k}(i)}=\prod_{0\leq s\leq \ell\leq p} (a^{(\ell)}_{-s})^{\mathbf{k}(\frac{\ell(\ell+1)}{2}+1+s)}\prod_{j=1}^{p}\left\{\phi_{0,j}^{+}(z)^{\alpha(\mathbf{k},j)}\,\phi_{0,j}^{-}(z)^{\beta(\mathbf{k},j)}\right\}
\]
and \eqref{eq:w0} follows. 
\end{proof}

\section{Random characteristic polynomials}\label{sec:rcp}

In this section we finally discuss our results in the random setting. Let $\mu_{k}$, $0\leq k\leq p$, be a collection of $p+1$ Borel probability measures with compact support in the complex plane. We consider $p+1$ sequences $(a_{n}^{(k)})_{n\in\mathbb{Z}}$, $0\leq k\leq p$, of complex random variables satisfying the following assumptions:
\begin{itemize}
\item[(a1)] For each $0\leq k\leq p$, the sequence $(a_{n}^{(k)})_{n\in\mathbb{Z}}$ is i.i.d. with distribution $\mu_{k}$.
\item[(a2)] The collection $\{a_{n}^{(k)}: n\in\mathbb{Z},\,\,0\leq k\leq p\}$ is jointly independent.
\item[(a3)] There exists a constant $C>0$ such that $|a_{n}^{(k)}|\leq C$ surely, for all $k$ and $n$.
\end{itemize}
We will keep these three assumptions throughout this section. All definitions given in Section~\ref{sec:tsop} apply.

In what follows we need the function
\begin{equation}\label{resolvW}
W(z):=\langle(zI-\mathcal{M})^{-1}\hat{e}_{0},\hat{e}_{0}\rangle=\frac{1}{z-a_{0}^{(0)}-\sum_{\ell=1}^{p}\sum_{k=0}^{\ell}a_{-k}^{(\ell)}\,\phi_{0,\ell-k}^{+}(z)\,\phi_{0,k}^{-}(z)}.
\end{equation}
This is exactly the function $w_{0}(z)$, see \eqref{def:funcwj}. Clearly, the functions $w_{j}(z)$ in \eqref{def:funcwj} are all analytic in the exterior of the closed disk centered at the origin with radius $\|\mathcal{M}\|<\infty$.   

\begin{theorem}\label{theo:rtso}
Let $(a_{n}^{(k)})_{n\in\mathbb{Z}}$, $0\leq k\leq p$, be $p+1$ sequences of complex-valued random variables satisfying conditions $(a1)$--$(a3)$ above. Let $\mathcal{M}$ be the two-sided operator defined in \eqref{def:opB}, with corresponding matrix $M=(m_{i,j})_{i,j\in\mathbb{Z}}$ defined in \eqref{def:matrixB}, and resolvent function \eqref{resolvW}. Let $M_{n}$ be the submatrix of $M$ with entries $m_{i,j}$, $-\lfloor(n-1)/2\rfloor\leq i,j\leq \lceil(n-1)/2\rceil$, with eigenvalues $\{\lambda_{k,n}\}_{k=1}^{n}$ (counting multiplicities), and let $Q_{n}(z)=\det(z I_{n}-M_{n})$. For each $s\in\mathbb{Z}_{\geq 0}$,
\begin{equation}\label{asympmoments}
\lim_{n\rightarrow\infty}\frac{1}{n}\mathbb{E}\left(\sum_{k=1}^{n} \lambda_{k,n}^{s}\right)=\mathbb{E}([W]_{s+1}).
\end{equation}
There exists $R>0$ such that for all $|z|>R$, 
\begin{equation}\label{asymppointwise}
\lim_{n\rightarrow\infty}\frac{1}{n}\mathbb{E}\left(\frac{Q_{n}'(z)}{Q_{n}(z)}\right)=\mathbb{E}(W(z)).
\end{equation}
\end{theorem} 
\begin{proof}
Note that $M_{n}$ is $n\times n$. Let $Q_{\ell}^{\pm}(z)$ be defined as in \eqref{def:charpol}. We have
\[
\frac{Q_{n}'(z)}{Q_{n}(z)}=\sum_{k=1}^{n}\frac{1}{z-\lambda_{k,n}}=\sum_{s=0}^{\infty}\left(\sum_{k=1}^{n}\lambda_{k,n}^{s}\right)\frac{1}{z^{s+1}},
\]
so
\[
\frac{1}{n}\left[\frac{Q_{n}'}{Q_{n}}\right]_{s+1}=\frac{1}{n}\sum_{k=1}^{n}\lambda_{k,n}^{s}=\frac{1}{n}\tr(M_{n}^{s}).
\]

Fix $0<\epsilon<1$. To simplify the notation, assume for the 
moment that $n$ is odd, say $n=2r+1$. We write 
\[
\tr(M_{n}^{s})=\sum_{j=-r}^{r}M_{n}^{s}(j,j)=\sum_{|j|\leq (1-\epsilon)r}M_{n}^{s}(j,j)+\sum_{(1-\epsilon)r<|j|\leq r}M_{n}^{s}(j,j).
\]
For each fixed $-r\leq j\leq r$, the entry $M_{n}^{s}(j,j)$ is the sum of all products $m_{i_{1},i_{2}}m_{i_{2},i_{3}}\cdots m_{i_{s},i_{s+1}}$ with $i_{1}=i_{s+1}=j$. Because of the banded structure of the matrix $M_{n}$, the number of such non-zero products remains bounded as $n$ increases (a simple bound is $(p+2)^{s}$), and so it follows from $(a3)$ that     
\begin{equation}\label{eq:estBns}
\frac{1}{n}\sum_{(1-\epsilon)r<|j|\leq r}|M_{n}^{s}(j,j)|\leq C_{p,s}\epsilon,
\end{equation}
where $C_{p,s}$ is a constant that depends only on $p$, $s$, and $C$ in $(a3)$.

In virtue of \eqref{eq:traceresol:1}, we also have, for every $-r\leq j\leq r$,
\[
\frac{Q_{r-j}^{+}(z) Q_{r+j}^{-}(z)}{Q_{n}(z)}=(zI_{n}-M_{n})^{-1}(j,j)=\sum_{s=0}^{\infty}\frac{M_{n}^{s}(j,j)}{z^{s+1}}
\]
so 
\begin{equation}\label{eq:Bnsc}
M_{n}^{s}(j,j)=\left[\frac{Q_{r-j}^{+} Q_{r+j}^{-}}{Q_{n}}\right]_{s+1}.
\end{equation}
Let $w_{j}(z)$ be the function defined in \eqref{def:funcwj}. According to \eqref{eq:estinf:3}, we have, for every $-r\leq j\leq r$,
\[
\frac{Q_{r-j}^{+}(z) Q_{r+j}^{-}(z)}{Q_{n}(z)}-w_{j}(z)=O\left(\frac{1}{z^{r-|j|+3+\lfloor(r-|j|)/p\rfloor}}\right).
\]
If $|j|\leq (1-\epsilon)r$, then $r-|j|\geq \epsilon r$, hence for all $n=2r+1$ large enough, we have
\begin{equation}\label{eq:error}
\frac{1}{n}\sum_{|j|\leq (1-\epsilon)r}\left[\frac{Q_{r-j}^{+} Q_{r+j}^{-}}{Q_{n}}-w_{j}\right]_{s+1}=0.
\end{equation}
So from \eqref{eq:Bnsc} and \eqref{eq:error} we deduce
\begin{equation}\label{eq:red}
\frac{1}{n} \sum_{|j|\leq (1-\epsilon)r}M_{n}^{s}(j,j)
=\frac{1}{n}\sum_{|j|\leq (1-\epsilon)r}\left[\frac{Q_{r-j}^{+} Q_{r+j}^{-}}{Q_{n}}\right]_{s+1}
=\frac{1}{n}\sum_{|j|\leq (1-\epsilon)r}[w_{j}]_{s+1}
\end{equation}
for all $n$ large enough. It follows easily from $(a1)$--$(a3)$ that each $[w_{j}]_{s+1}$ is integrable, and $\mathbb{E}([w_{j}]_{s+1})=\mathbb{E}([W]_{s+1})$ for every $j$. So we conclude, applying \eqref{eq:estBns} and \eqref{eq:red}, that
\begin{align*}
\frac{1}{n}\mathbb{E}\left(\tr{(M_{n}^{s})}\right) & =\frac{1}{n}\mathbb{E}\Big(\sum_{|j|\leq (1-\epsilon)r}M_{n}^{s}(j,j)\Big)+\frac{1}{n}\mathbb{E}\Big(\sum_{(1-\epsilon)r<|j|\leq r}M_{n}^{s}(j,j)\Big)\\
& =\frac{1}{n}(1+2\lfloor(1-\epsilon)r\rfloor)\,\mathbb{E}([W]_{s+1})+O(\epsilon).
\end{align*}
Letting $n=2r+1\rightarrow\infty$ first and then $\epsilon\rightarrow 0$, we get \eqref{asympmoments}. The case $n$ even is handled in identical manner, using \eqref{eq:estinf:4}.

By $(a3)$, there exists $R>0$ such that $\|M_{n}\|, \|\mathcal{M}\|\leq R$ for all $n$. Then, for $|z|>R$, a simple argument shows that \eqref{asymppointwise} follows from \eqref{asympmoments}.
\end{proof}

\begin{remark}
By the i.i.d. assumptions, it is clear that in Theorem~\ref{theo:rtso} we could have considered other sequences of matrices $M_{n}$ with the same asymptotic properties. For example, we could have taken the matrices $M_{n}=(m_{i,j})_{1\leq i,j\leq n}$. With this choice of the matrices $M_{n}$, Theorem~\ref{theo:main:intro} is immediately justified. Indeed, if the have the collections $\mathcal{A}=(a^{(0)},\ldots,a^{(p)})$, $\mathcal{B}=(b^{(0)},\ldots,b^{(p)})$, and $\alpha=(\alpha_{j}^{(k)})_{0\leq j\leq k\leq p}$ as in the statement of Theorem~\ref{theo:main:intro}, we can extend the sequences $a^{(k)}=(a_{n}^{(k)})_{n=1}^{\infty}$ to sequences $(a_{n}^{(k)})_{n\in\mathbb{Z}}$ defining for $n\leq 0$ the terms
\[
a_{n}^{(k)}=\begin{cases}
\alpha_{-n}^{(k)} & \mbox{if}\,\,-k\leq n\leq 0,\\[0.4em]
b_{-n-k}^{(k)} & \mbox{if}\,\,\,\, n\leq -k-1.
\end{cases}
\] 
Then the function \eqref{resolvW} coincides with \eqref{funcW:intro}, and Theorem~\ref{theo:rtso} applied to the sequence of matrices $M_{n}=(m_{i,j})_{1\leq i,j\leq n}$ implies Theorem~\ref{theo:main:intro}.
\end{remark}

In the rest of this section, we consider the probability distribution of the random vector $(\phi_{0,1}^{+}(z),\ldots,\phi_{0,p}^{+}(z))$, and describe certain relations between this distribution, the function $\mathbb{E}(W(z))$, and the measures $\mu_{k}$, $0\leq k\leq p$. The relations we describe extend some results obtained in \cite{LopPro} in the case $p=1$. 

\begin{definition}
Let $\phi_{0,k}^{+}$, $1\leq k\leq p$, be the resolvent functions defined in \eqref{def:resolfunc}, associated with the operator $\mathcal{H}_{0}^{+}$. Let $\sigma_{z}$ denote the joint probability distribution of the vector $(\phi_{0,1}^{+}(z),\ldots,\phi_{0,p}^{+}(z))$. Given $\mathbf{n}=(n_{1},\ldots,n_{p})\in\mathbb{Z}_{\geq 0}^{p}$, let
\begin{equation}\label{def:gfunc}
g_{\mathbf{n}}(z):=\mathbb{E}\Bigg(\prod_{k=1}^{p}\phi_{0,k}^{+}(z)^{n_{k}}\Bigg).
\end{equation}
\end{definition}

\begin{definition}\label{def:mul}
For $\mu_{\ell}$ the common distribution of the random variables $(a_{n}^{(\ell)})_{n\in\mathbb{Z}}$, we define the moments
\begin{equation}\label{eq:moments}
m_{k}^{(\ell)}:=\int z^{k}\,d\mu_{\ell}(z),\qquad k\geq 0.
\end{equation}
\end{definition}

\begin{theorem}
The following identity holds:
\begin{equation}\label{eq:EW}
\mathbb{E}(W(z))=\sum_{r=0}^{\infty}
\sum_{\mathbf{k}\in \widehat{C}(r)}\frac{1}{z^{r+1}}\binom{r}{\mathbf{k}}\Big(\prod_{0\leq s\leq \ell\leq p} m_{\mathbf{k}(\frac{\ell(\ell+1)}{2}+1+s)}^{(\ell)}\Big)\,g_{\alpha(\mathbf{k})}(z)\,g_{\beta(\mathbf{k})}(z),
\end{equation}
where $\alpha(\mathbf{k})$ and $\beta(\mathbf{k})$ are defined in \eqref{def:alphabeta:2}, and $\mathbf{k}(i)$ is the $i$-th component of the vector $\mathbf{k}$.
\end{theorem}
\begin{proof}
The identity follows by taking expectation in \eqref{eq:w0}. It follows from assumptions (a1) and (a2) that the variables $a_{-s}^{(\ell)}$ in \eqref{eq:w0} are independent of the functions $\phi_{0,j}^{+}$, $\phi_{0,j}^{-}$, and the vectors $(\phi_{0,1}^{+},\ldots,\phi_{0,p}^{+})$ and $(\phi_{0,1}^{-},\ldots,\phi_{0,p}^{-})$ are independent and have the same distribution. Thus, by \eqref{def:gfunc} and \eqref{eq:moments}, we obtain
\begin{gather*}
\mathbb{E}\left(\prod_{0\leq s\leq \ell\leq p} (a^{(\ell)}_{-s})^{\mathbf{k}(\frac{\ell(\ell+1)}{2}+1+s)}\prod_{j=1}^{p}\left\{\phi_{0,j}^{+}(z)^{\alpha(\mathbf{k},j)}\,\phi_{0,j}^{-}(z)^{\beta(\mathbf{k},j)}\right\}\right)\\
=\Big(\prod_{0\leq s\leq \ell\leq p} m_{\mathbf{k}(\frac{\ell(\ell+1)}{2}+1+s)}^{(\ell)}\Big)\,g_{\alpha(\mathbf{k})}(z)\,g_{\beta(\mathbf{k})}(z),
\end{gather*}
which yields \eqref{eq:EW}.
\end{proof}

Given vectors $\mathbf{r}=(r_{1},\ldots,r_{p})\in\mathbb{Z}_{\geq 0}^{p}$ and $\mathbf{k}=(k_{0},k_{1},\ldots,k_{p})\in\mathbb{Z}_{\geq 0}^{p+1}$, we define
\begin{equation}\label{def:etavec}
\eta(\mathbf{r},\mathbf{k}):=(k_{1}+r_{2}, k_{2}+r_{3},\ldots,k_{p-1}+r_{p},k_{p})\in\mathbb{Z}_{\geq 0}^{p}.
\end{equation}

\begin{theorem}
For any vector $\mathbf{r}=(r_{1},\ldots,r_{p})\in\mathbb{Z}_{\geq 0}^{p}$, we have
\[
g_{\mathbf{r}}(z)=\sum_{n=0}^{\infty}\sum_{\mathbf{k}\in C(n)}\frac{1}{z^{n+|\mathbf{r}|}}\binom{n+|\mathbf{r}|-1}{|\mathbf{r}|-1}\binom{n}{\mathbf{k}}\Big(\prod_{\ell=0}^{p}m_{\mathbf{k}(\ell)}^{(\ell)}\Big) g_{\eta(\mathbf{r},\mathbf{k})}(z),
\]
where $|\mathbf{r}|=\sum_{j=1}^{p}r_{j}$, see also \eqref{def:Cn}.
\end{theorem}
\begin{proof}
According to \eqref{eq:phiphi1},
\[
\prod_{j=1}^{p}\phi_{0,j}^{+}(z)^{r_{j}}=\sum_{n=0}^{\infty}\sum_{\mathbf{k}\in C(n)}\frac{1}{z^{n+|\mathbf{r}|}}
\binom{n+|\mathbf{r}|-1}{|\mathbf{r}|-1}\binom{n}{\mathbf{k}}\Big(\prod_{j=0}^{p} (a_{1}^{(j)})^{\mathbf{k}(j)}\Big)\prod_{j=1}^{p}\phi_{1,j}^{+}(z)^{\mathbf{k}(j)+r_{j+1}}
\]
where $r_{p+1}=0$. Taking expectation, we have
\[
\mathbb{E}\Big(\Big(\prod_{j=0}^{p} (a_{1}^{(j)})^{\mathbf{k}(j)}\Big)\prod_{j=1}^{p}\phi_{1,j}^{+}(z)^{\mathbf{k}(j)+r_{j+1}}\Big)=\Big(\prod_{\ell=0}^{p}m_{\mathbf{k}(\ell)}^{(\ell)}\Big) g_{\eta(\mathbf{r},\mathbf{k})}(z),
\]
so the result follows.
\end{proof}

\begin{definition}
Define the set 
\[
\mathcal{D}_{z}:=\{((t_{0},\ldots,t_{p}),(x_{1},\ldots,x_{p}))\in\mathbb{C}^{p+1}\times \mathbb{C}^{p}: z-t_{0}-\sum_{k=1}^{p}t_{k} x_{k}\neq 0\}.
\]
We will write $\mathbf{t}=(t_{0},\ldots,t_{p})$, $\mathbf{x}=(x_{1},\ldots,x_{p})$. Let $\lambda_{z}:\mathcal{D}_{z}\longrightarrow\mathbb{C}^{p}$ be the function
\[
\lambda_{z}(\mathbf{t},\mathbf{x}):=\left(\frac{1}{z-t_{0}-\sum_{k=1}^{p}t_{k} \,x_{k}},\frac{x_{1}}{z-t_{0}-\sum_{k=1}^{p}t_{k}\,x_{k}},\ldots,\frac{x_{p-1}}{z-t_{0}-\sum_{k=1}^{p}t_{k}\,x_{k}}\right).
\]
The product measure $\mu:=\mu_{0}\times\cdots\times\mu_{p}$ is the probability distribution of $(a_{1}^{(0)},\ldots,a_{1}^{(p)})$, in accordance with (a1). 
\end{definition}

We have the following \emph{invariance principle}.

\begin{theorem}
For any bounded measurable function or continuous function $f:\mathbb{C}^{p}\longrightarrow\mathbb{C}$, and for all $z$ large enough, we have
\begin{equation}\label{eq:invprop}
\int_{\mathbb{C}^{p}} f(\mathbf{x})\,d\sigma_{z}(\mathbf{x})
=\iint_{\mathcal{D}_{z}} f(\lambda_{z}(\mathbf{t},\mathbf{x}))\,d\mu(\mathbf{t})\,d\sigma_{z}(\mathbf{x}).
\end{equation}
\end{theorem}
\begin{proof}
Let $\Phi(z):=(\phi_{0,1}^{+}(z),\ldots,\phi_{0,p}^{+}(z))$ and $\Phi_{1}(z):=(\phi_{1,1}^{+}(z),\ldots,\phi_{1,p}^{+}(z))$. By (a3) and Lemma~\ref{lem:relfuncphi}, we know there exists $R>0$ such that for all $|z|>R$, we have $z-a_{1}^{(0)}-\sum_{k=1}^{p} a_{1}^{(k)} \phi_{1,k}^{+}(z)\neq 0$, so $((a_{1}^{(0)},\ldots,a_{1}^{(p)}),\Phi_{1}(z))\in\mathcal{D}_{z}$, surely. It is clear that the vectors $\Phi(z)$ and $\Phi_{1}(z)$ have the same distribution $\sigma_{z}$, and $(a_{1}^{(0)},\ldots,a_{1}^{(p)})$ and $\Phi_{1}(z)$ are independent (see definition of $\mathcal{H}_{1}^{+}$), so the distribution of $((a_{1}^{(0)},\ldots,a_{1}^{(p)}),\Phi_{1}(z))$ is $\mu\times \sigma_{z}$, and this is a measure on $\mathcal{D}_{z}$. 

The relations \eqref{rel:phis:1}--\eqref{rel:phis:2} imply $\Phi(z)=\lambda_{z}((a_{1}^{(0)},\ldots,a_{1}^{(p)}),\Phi_{1}(z))$, so $\sigma_{z}=(\lambda_{z})_{*}(\mu\times\sigma_{z})$ is the push-forward measure of $\mu\times\sigma_{z}$ under $\lambda_{z}$. By the change of variable formula, for any bounded measurable function $f$ on $\mathbb{C}^{p}$, we have
\[
\int f\,d\sigma_z=\iint f\circ\lambda_{z}\,d(\mu\times\sigma_{z}).
\]
Since $\Phi(z)$ is bounded, the measure $\sigma_{z}$ has bounded support, so any continuous function $f$ is in $L^{1}(\sigma_{z})$, and \eqref{eq:invprop} is also valid for such functions.
\end{proof}

\bigskip

\noindent \textsc{Department of Mathematics, University of Central Florida, 4393 Andromeda Loop North, Orlando, FL 32816} \\
\textit{Email address}: \texttt{abey.lopez-garcia\symbol{'100}ucf.edu}

\bigskip

\noindent \textsc{Department of Mathematics and Statistics, University of South Alabama, 411 University Boulevard North, 
Mobile, AL 36688}\\
\textit{Email address}: \texttt{prokhoro\symbol{'100}southalabama.edu}

\end{document}